\documentclass{amsart}
\usepackage{amsmath, amssymb, amscd,amsthm,amsfonts}
\usepackage{orcidlink}
\usepackage{makecell} 
\usepackage{hyperref}
\hypersetup{
 colorlinks=true, 
 linkcolor=blue, 
 citecolor=green, 
 urlcolor=magenta 
}
\usepackage{tikz-cd}
\usepackage{tikz}
\usetikzlibrary{positioning}
\usepackage{microtype}
\theoremstyle{plain}
\newtheorem{theorem}{Theorem}[section]
\newtheorem{lemma}[theorem]{Lemma}
\newtheorem{proposition}[theorem]{Proposition}
\theoremstyle{definition}

\newtheorem{example}[theorem]{Example}
\newtheorem{remark}[theorem]{Remark}
\newtheorem{corollary}[theorem]{Corollary}
\title{The Euler Characteristic of Finite Subset Spaces }
\author{Walid Taamallah \orcidlink{0009-0006-8723-0264} }
\address{Institut Preparatoire aux Etudes d'Ingenieurs El Manar, Tunisie}
\email{walid.taamallah@ipeiem.utm.tn}
\date{\today}
\newcommand{\sub}[1]{\operatorname{Sub}_{#1}}
\newcommand{\symp}[1]{\operatorname{SP}^{#1}}
\begin{document}
\begin{abstract}
For a topological space $X$, the space of finite subsets $\sub{n}X$ consists of non-empty subsets of $X$ of cardinality at most $n$. We compute the Euler characteristic of $\sub{n}X$ for any space $X$ having the homotopy type of a finite CW complex. We obtain the explicit formula $\chi(\sub{n}X) = \sum_{i=1}^n \binom{\chi(X)}{i}$.
\end{abstract}
\maketitle
\section{Introduction}
For a topological space $X$, let $\sub{n}X$ denote the space of non-empty subsets of $X$ of cardinality at most $n$:
\[
\sub{n} X = \{A \subset X \mid 0 < |A| \le n\}. 
\] 
The topology on $\sub{n}X$ is inherited as a quotient space induced by the natural identification map:
\[
\pi : X^n \to \sub{n}X, \qquad (x_1,\ldots,x_n) \longmapsto \{x_1,\ldots, x_n\}. 
\]
The functors $\sub{n}(-)$ preserve homotopy type.

First introduced by Borsuk and Ulam~\cite{borsukulam}, these spaces have been widely studied in both algebraic topology and algebraic geometry, where they are often referred to as Ran spaces. The study of these spaces includes investigations into their connectivity properties~\cite{tuffley1,kallelsjerve, mostovoysadykov, taamallah}, their rational homotopy types~\cite{felixtanre, mostovoy}, and their homology for specific classes of spaces, such as graphs, circles, and surfaces, or for small values of $n$ in $\sub{n}X$~\cite{tuffley3, tuffley2, tuffley4, taamallah}.

More recently, various interesting papers have appeared on the subject~\cite{douteaulabeye, mostovoy, lazovskis}. While these developments have advanced our understanding of the topological and homological properties of finite subset spaces, an explicit formula for their most classical topological invariant (the Euler characteristic) has remained absent from the literature. We give the following short formula, valid for any space having the homotopy type of a finite CW complex.

\begin{theorem}
Let $X$ be any space of the homotopy type of a finite CW complex. The Euler characteristic of the $n$th finite subset space of $X$ is given by:
\begin{equation*}
\chi(\sub{n}X) = \sum_{i=1}^n \binom{\chi(X)}{i}.
\end{equation*}
\end{theorem}
Our method consists of constructing $\sub{n}X$ inductively by means of pushouts involving the fat diagonal subspaces in the symmetric products $\symp{k}X$, for $k\leq n$. 

Our computation is a direct follow-up to~\cite{taamallah}.
\section{Preliminaries}\label{prelim}

Throughout, $X$ will be a finite CW complex. We will represent
a strict pushout $X = B \cup_A C$ by a commuting square:
\[
\begin{tikzcd}
A \arrow[r,"f"] \arrow[d,"g"']&B \arrow[d,"\alpha"]\\
C\arrow[r,"\beta"']&X
\end{tikzcd}
\]
The Euler characteristic $\chi$ is additive with respect to such pushouts: i.e., 
\[
\chi(X) = \chi(B) + \chi(C) - \chi(A). \]
 Denote by $\symp{n}X$ the $n$th symmetric product of $X$, which is defined as the quotient $\symp{n}X=X^n/\Sigma_n$ by the action of the $n$th symmetric group. The fat diagonal $F_2(X,n)$ is the subspace of $X^n$ consisting of $n$-tuples with at least $2$ equal entries. Its image in $\symp{n}X$ is denoted $B_2(X,n)$ and is called the unordered fat diagonal.
The following lemma is the basis of our computation.
\begin{lemma}\label{pushoutsubn} There is a pushout diagram
\[
\begin{tikzcd}
B_2(X, n) \arrow[r, "f_n"] \arrow[d, "g_n"'] & \symp{n} X \arrow[d, "\alpha_n"] \\
\sub{n-1} X \arrow[r, "\beta_n"'] & \sub{n} X
\end{tikzcd}
\]
where $f_n$ and $\beta_n$ are inclusions, $\alpha_n$ is the quotient map, and $g_n$ is the map induced by restricting $\alpha_n$ to $B_2(X,n)$, whose image lands in $\sub{n-1}X$.
\end{lemma}
The Euler characteristic of the symmetric product is known and given by Macdonald's formula~\cite{macdonald}
\[
\chi(\symp{n}X) = \binom{\chi(X)+n-1}{n}.
\]
The crux of the proof is to determine $\chi (B_2(X,n))$ (this is done in Proposition \ref{chiB2}) to obtain our calculation. 
\section{Euler characteristic of the fat diagonal}Let $X$ be a finite CW complex, and denote its Euler characteristic by $\chi(X) = k$. Recall from Section \ref{prelim} that the fat diagonal $F_2(X, n)$ is the subspace of the cartesian product $X^n$ consisting of tuples where at least two coordinates are equal. We express this space as a finite union of closed diagonal hyperplanes:
\[ 
F_2(X, n) = \bigcup_{1 \le i < j \le n} \Delta_{(i,j)}, 
\]
where $\Delta_{(i,j)} = \{(x_1, \dots, x_n) \in X^n \mid x_i = x_j\}$. 

For a finite union of compact closed subspaces, the inclusion-exclusion principle for the singular Euler characteristic yields the following lemma.
\begin{lemma}
The Euler characteristic of the fat diagonal satisfies
\begin{equation}\label{F2one}
\chi\Big(F_2(X, n)\Big) = -\sum_{\emptyset \neq S \subseteq E(K_n)} (-1)^{|S|} \chi(\Delta_S),
\end{equation}
where $S$ ranges over all non-empty subsets of the edge set $E(K_n) = \{(i,j) \mid 1 \le i < j \le n\}$ of the complete graph $K_n$, and the intersection space is defined as
\[
\Delta_S = \bigcap_{(i,j) \in S} \Delta_{(i,j)}.
\]
\end{lemma}
\begin{proof}
Applying the topological inclusion-exclusion formula yields
\[
\chi\Big(F_2(X, n)\Big) = \sum_{r=1}^{\binom{n}{2}} \left( (-1)^{r-1} \sum_{S \subseteq E(K_n), \, |S| = r} \chi(\Delta_S) \right).
\]
Changing the indexing to range over all non-empty edge subsets $S \subseteq E(K_n)$ completes the proof.
\end{proof}To evaluate the right-hand side of Equation~\eqref{F2one}, we group identical intersection spaces using the partition poset $\Pi_n$ of the index set $[n] = \{1, \dots, n\}$. Every edge subset $S \subseteq E(K_n)$ uniquely defines a spanning subgraph $G = ([n], S)$ of the complete graph $K_n$. We denote by $\lambda(S) \in \Pi_n$ the partition whose blocks correspond to the connected components of $G$. 

A point $(x_1, \dots, x_n) \in X^n$ belongs to $\Delta_S$ if and only if $x_i = x_j$ whenever $i$ and $j$ lie in the same connected component of $G$. This implies a homeomorphism $\Delta_S \cong X^{|\lambda(S)|}$, where $|\lambda(S)|$ is the number of blocks in $\lambda(S)$. Consequently, we have $\chi(\Delta_S) = k^{|\lambda(S)|}$. Collecting identical partition terms allows us to rewrite Equation~\eqref{F2one} as a sum over the partition lattice:
\begin{equation}\label{F2two}
\chi\Big(F_2(X, n)\Big) = -\sum_{\lambda > \hat{0}} c(\lambda) k^{|\lambda|}, \quad \text{where} \quad c(\lambda) = \sum_{S \implies \lambda} (-1)^{|S|}.
\end{equation}
Here, the notation $S \implies \lambda$ indicates that the connected components of the graph $([n], S)$ match the blocks of $\lambda$, and $\hat{0}$ denotes the partition $\hat{0}=\Big\{\{1\},\dots,\{n\}\Big\}$.
\begin{remark}
The trivial partition $\lambda = \hat{0} = \Big\{\{1\}, \dots, \{n\}\Big\}$ is excluded from the summation because any edge subset $S$ must contain at least one edge ($|S| \ge 1$), so every partition $\lambda $ must contain at least one block of size $\ge 2$.
\end{remark}
\begin{lemma}\label{comb_coeff}
For any partition $\lambda \in \Pi_n$, the combinatorial coefficient $c(\lambda)$ satisfies
\begin{equation}
c(\lambda) = (-1)^{n - |\lambda|} \prod_{B \in \lambda} (|B|-1)! = \mu(\hat{0}, \lambda),
\end{equation}
where $\mu(\hat{0}, \lambda)$ is the M\"{o}bius function of the partition lattice $\Pi_n$.
\end{lemma}
\begin{proof}
The condition $S\implies\lambda$ requires that $S$ contains no cross-edges between distinct blocks of $\lambda$. Thus, $S$ decomposes uniquely as a disjoint union of local edge sets $ S = \bigsqcup_{B \in \lambda} S_B$, where each $S_B \subseteq E(K_B)$ forms a connected spanning subgraph on the vertex block $B$, where $K_B$ denotes the complete graph with vertex set $B$. This structure allows the global sum to factor into a product of independent local sums:
\[
c(\lambda) = \prod_{B \in \lambda} \left( \sum_{\substack{S_B \subseteq E(K_B) \\ (B, S_B) \text{ is connected}}} (-1)^{|S_B|} \right).
\]
By a classic theorem on chromatic polynomials~\cite{morganvena}, the inner sum over connected spanning subgraphs is precisely the linear coefficient of the chromatic polynomial of the complete graph $K_B$. Since 
\[
P(K_B; t) = t(t - 1)\cdots (t - |B| + 1),
\] 
extracting this linear coefficient yields $(-1)^{|B|-1}(|B|-1)!$.

Multiplying these terms together gives
\[ 
c(\lambda) = \prod_{B \in \lambda} (-1)^{|B|-1}(|B|-1)! = (-1)^{\sum_{B \in \lambda} (|B|-1)} \prod_{B \in \lambda} (|B|-1)! 
\]
Since 
\[
\sum_{B \in \lambda} |B| = n \;\text{and}\;\sum_{B \in \lambda} 1 = |\lambda|,
\] the exponent simplifies to $n - |\lambda|$. Hence,
\[
c(\lambda) = (-1)^{n - |\lambda|} \prod_{B \in \lambda} (|B|-1)! 
\]
matching the known value of the lattice M\"{o}bius function $\mu(\hat{0}, \lambda)$~\cite{callanstong}.
\end{proof}
\begin{example}
To illustrate the combinatorial structure of Equation~\eqref{F2two} and Lemma~\ref{comb_coeff}, consider the case $n = 3$. The partition poset $\Pi_3$ contains four partitions strictly greater than $\hat{0}$:
\begin{itemize}
\item Three partitions $\lambda_1=\big\{\{1,2\}, \{3\}\big\}$, $\lambda_2=\big\{\{2,3\}, \{1\}\big\}$, and $\lambda_3=\big\{\{1,3\}, \{2\}\big\}$, each possessing a unique generating edge $S \subset E(K_3)$. This yields $c(\lambda_i) = (-1)^1 = -1$ for $i\in \{1,2,3\}$.
\item The maximal partition $\lambda_4 = \big\{\{1,2,3\}\big\}$, which is spanned by the connected subgraphs of $K_3$. These consist of three paths of length two, corresponding to $S = \{(1,2), (2,3)\}$, $S = \{(2,3), (1,3)\}$, and $S = \{(1,2), (1,3)\}$, and the complete graph $K_3$ itself, corresponding to $S = \{(1,2), (2,3), (1,3)\}$. This yields $c(\lambda_4) = 3(-1)^2 + 1(-1)^3 = 2$. 
\end{itemize}These evaluations coincide with the lattice M\"{o}bius function values from Lemma~\ref{comb_coeff}, where $\mu(\hat{0}, \lambda_i) = (-1)^{3-2}(2-1)!(1-1)! = -1$ for $i \in \{1,2,3\}$, and $\mu(\hat{0}, \lambda_4) = (-1)^{3-1}(3-1)! = 2$.

Evaluating Equation~\eqref{F2two} using the respective block counts $|\lambda_i| = 2$ for $i\in \{1,2,3\}$ and $|\lambda_4| = 1$, we obtain:
\[
\chi\Big(F_2(X, 3)\Big) = -\Big( 3(-1)k^2 + 2k \Big) = 3k^2 - 2k.
\]
This matches the expression $k^3 - k(k-1)(k-2)$ established below in Proposition~\ref{chiF2}.
\end{example}
With these combinatorial coefficients established, we can compute the global sum across the entire partition lattice.
\begin{proposition}\label{chiF2}
Let $X$ be a finite CW complex with $\chi(X)=k$. The Euler characteristic of its fat diagonal is given by
\begin{equation*}
\chi\big(F_2(X, n)\big) = k^n - k(k-1)\cdots(k-n+1).
\end{equation*}
\end{proposition}
\begin{proof}
Substituting the evaluation from Lemma~\ref{comb_coeff} into Equation~\eqref{F2two} yields
\[ 
\chi\big(F_2(X, n)\big) = -\sum_{\lambda > \hat{0}} \mu(\hat{0}, \lambda) k^{|\lambda|}.
\]
To evaluate this sum across the entire lattice, we add and subtract the contribution of the missing minimal element $\lambda = \hat{0}$:
\[
\chi\big(F_2(X, n)\big) = \left( -\sum_{\lambda \in \Pi_n} \mu(\hat{0}, \lambda) k^{|\lambda|} \right) + \mu(\hat{0}, \hat{0}) k^{|\hat{0}|}.
\]
Recall that $\mu(\hat{0}, \hat{0}) = 1$, and the minimal partition $\hat{0}$ contains exactly $n$ singletons, meaning $|\hat{0}| = n$. Rearranging the terms yields:
\begin{equation*}
\chi\big(F_2(X, n)\big) = k^n - \sum_{\lambda \in \Pi_n} \mu(\hat{0}, \lambda) k^{|\lambda|}. 
\end{equation*}
The remaining sum matches the definition of the characteristic polynomial $P(\Pi_n; k)$ of the partition lattice. Invoking its known factorization identity~\cite[p.~2]{hallamsagan} yields \[P(\Pi_n; k) = k(k-1)\cdots(k-n+1),\] which completes the proof.
\end{proof}
This calculation serves as the main tool to track the Euler characteristic of the unordered fat diagonal $B_2(X,n)$ and the finite subset space $\sub{n}X$ in the sections that follow.
 \section{The Euler characteristic of the unordered fat diagonal}
 
 Recall that for a finite group $G$ acting on a finite CW complex $Y$, the Euler characteristic of the orbit space $Y/G$ is given by the standard formula (see \cite{hirzebruchhofer, bryanfulman}):
\begin{equation}\label{burnside}
 \chi(Y/G) = \frac{1}{|G|} \sum_{g \in G} \chi(Y^g)
\end{equation}
 where $Y^g$ denotes the fixed-point subspace of $Y$ under the action of $g$ 
 \[
 Y^g = \{y \in Y \mid g \cdot y = y\}.
 \] 
 Let $\sigma \in \Sigma_n$ be any permutation other than the identity ($\sigma \neq \mathrm{id}$). Then $\sigma$ must contain at least one cycle of length $d \ge 2$. 
 
 For a point $(x_1,\dots,x_n) \in X^n$ to be fixed by $\sigma$, all coordinates in that cycle must be equal. Since $d \ge 2$, the point $(x_1,\dots,x_n)$ necessarily belongs to the fat diagonal $F_2(X, n)$. Therefore, $(X^n)^\sigma \subseteq F_2(X, n)$. Since, by definition, 
\[
F_2(X,n)^\sigma = F_2(X,n) \cap (X^n)^\sigma
\]
 we have 
\[
F_2(X,n)^\sigma = (X^n)^\sigma. 
\] 
 Hence, 
 \begin{equation}\label{B2one}
 \chi(F_2(X, n)^\sigma) = \chi((X^n)^\sigma) \quad \text{for all } \sigma \neq \mathrm{id}.
 \end{equation}
 Applying the formula \eqref{burnside} to the quotient $B_2(X, n) = F_2(X, n)/\Sigma_n$:
 \[
 \chi(B_2(X, n)) = \frac{1}{n!} \left( \chi(F_2(X, n)) + \sum_{\sigma \neq \mathrm{id}} \chi(F_2(X, n)^\sigma) \right).
 \]
 Using \eqref{B2one}, one gets 
 \begin{equation}\label{B2two}
 \chi(B_2(X, n)) = \frac{1}{n!} \left( \chi(F_2(X, n)) + \sum_{\sigma \neq \mathrm{id}} \chi((X^n)^\sigma) \right). 
 \end{equation}
 Applying the formula \eqref{burnside} to the entire symmetric product $\symp{n} X = X^n / \Sigma_n$:
 \[
 n! \chi(\symp{n} X) = \chi((X^n)^{\mathrm{id}}) + \sum_{\sigma \neq \mathrm{id}} \chi((X^n)^\sigma).
 \]
Substituting $\chi((X^n)^{\mathrm{id}}) = k^n$, we isolate the sum:
\begin{equation}\label{B2three}
 \sum_{\sigma \neq \mathrm{id}} \chi((X^n)^\sigma) = n! \chi(\symp{n} X) - k^n.
\end{equation} 
Now, substituting the evaluation of $\chi(F_2(X,n))$ from Proposition~\ref{chiF2} along with~\eqref{B2three} into~\eqref{B2two} yields
\[
\chi(B_2(X, n)) = \frac{1}{n!} \Big( [k^n - k(k-1)\cdots(k-n+1)] + [n! \chi(\symp{n} X) - k^n] \Big).
\]
This simplifies to
\[
\chi(B_2(X, n)) = \chi(\symp{n} X) - \frac{k(k-1)\cdots(k-n+1)}{n!}.
\]
Invoking Macdonald's formula~\cite{macdonald} alongside the identity for binomial coefficients yields the following desired result.
\begin{proposition}\label{chiB2}
 For a finite CW complex $X$ with $\chi(X)=k$, we have 
\[
\chi\Big(B_2(X, n)\Big) = \binom{k+n-1}{n} - \binom{k}{n}.
\]
 \end{proposition}
\section{The Euler characteristic of the finite subset space}
 
In this final section, we prove our main result. 
\begin{theorem}
Let $X$ be a space of the homotopy type of a finite CW complex 
\[
\chi(\sub{n} X) = \sum_{i=1}^n \binom{\chi(X)}{i}. 
\]
\end{theorem}
\begin{proof}
By the homotopy invariance of the functor $\sub{n}(-)$, it suffices to assume that $X$ is a finite CW complex. Using the additive property of the Euler characteristic on the pushout square from Lemma~\ref{pushoutsubn}:
\[
\chi(\sub{n} X) = \chi(\sub{n-1} X) + \chi(\symp{n} X) - \chi(B_2(X, n)).
\] 
Using $k = \chi(X)$, we obtain:
\[
\chi(\symp{n} X)-\chi(B_2(X, n)) = \binom{k+n-1}{n}- \left[\binom{k+n-1}{n} - \binom{k}{n}\right] =\binom{k}{n}.
\]
Hence,
\begin{equation}\label{chisubninduction}
 \chi(\sub{n} X) = \chi(\sub{n-1} X)+\binom{k}{n}.
 \end{equation}
 We now proceed by induction on $n$.
 
For $n=1$, we have $\sub{1} X \cong X$, which implies 
\[
\sum_{i=1}^1 \binom{\chi(X)}{i} = \chi(X).
\]
Thus, the base case holds. 

Assume the formula holds true for $n-1$:
 \[
 \chi(\sub{n-1} X) = \sum_{i=1}^{n-1} \binom{k}{i}.
 \]
 Applying the induction hypothesis to the recurrence relation \eqref{chisubninduction} yields
 \[
\chi(\sub{n} X) = \sum_{i=1}^{n-1} \binom{k}{i} + \binom{k}{n} = \sum_{i=1}^n \binom{k}{i} = \sum_{i=1}^n \binom{\chi(X)}{i}.
\]
\end{proof}
As a corollary, we recover a known calculation of Tuffley (\cite{tuffley4}, Theorem 5)
\begin{corollary}
 Let $\Sigma_g$ be a closed orientable surface of genus $g$. Then
 \[
\chi(\sub{3}\Sigma_g)=\dfrac{-4g^3+12g^2-17g+9}{3}.
\]
\end{corollary}

\end{document}